\documentclass[11pt]{article}
\usepackage[latin1]{inputenc}

\usepackage{amsmath}
\usepackage{amsfonts}
\usepackage{amssymb}
\usepackage{amsthm}
\usepackage{graphics}
\usepackage{mathabx}
\usepackage[all]{xy}
\usepackage{hyperref}

\newtheorem{Prop}{Proposition}

\newtheorem{Lem}[Prop]{Lemma}
\newtheorem{Thm}[Prop]{Theorem}
\newtheorem{Cor}[Prop]{Corollary}

\newcommand{\R}{\mathbb{R}}
\newcommand{\C}{\mathbb{C}}
\renewcommand{\H}{\mathbb{H}}

\newcommand{\Z}{\mathbb{Z}}
\newcommand{\E}{\mathcal{E}}
\newcommand{\tr}{\mathrm{tr\,}}

\newcommand{\<}{\langle}
\renewcommand{\>}{\rangle}
\renewcommand{\d}{\partial}
\newcommand{\Isom}{\mathrm{Isom}}

\date{}

\title{Deformations of hyperbolic convex polyhedra and cone-$3$-manifolds}
\author{Gr\'egoire Montcouquiol\footnote{Partially supported by the ANR program GeomEinstein 06-BLAN-0154. \newline Univ Paris-Sud, Laboratoire de Math\'ematiques, UMR 8628, Orsay, F-91405;
CNRS, Orsay, F-91405.}}
\begin{document}

\maketitle

\begin{minipage}{13.8cm}
\small

\noindent{\bf Abstract}

\vskip 0.5\baselineskip
\noindent The Stoker problem, first formulated in \cite{Stoker}, consists in understanding to what extent a convex polyhedron is determined by its dihedral angles. By means of the double construction, this problem is intimately related to rigidity issues for $3$-dimensional cone-manifolds. In \cite{Maz-GM}, two such rigidity results were proven, implying that the infinitesimal version of the Stoker conjecture is true in the hyperbolic and Euclidean cases. In this second article, we show that local rigidity holds and prove that the space of convex hyperbolic polyhedra with given combinatorial type is locally parametrized by the set of dihedral angles, together with a similar statement for hyperbolic cone-$3$-manifolds. 

\vskip 0.5\baselineskip
\noindent{\bf R\'esum\'e}

\vskip 0.5\baselineskip
\noindent Le problème de Stoker, formulé pour la première fois dans \cite{Stoker}, consiste à comprendre dans quelle mesure un polyèdre convexe est déterminé par ses angles dièdres. Via la construction qui à un polyèdre associe son double, ce problème est intimement lié à des questions de rigidité pour les cônes-variétés de dimension $3$. Dans \cite{Maz-GM}, deux tels résultats de rigidité ont été prouvés, impliquant que la version infinitésimale de la conjecture de Stoker est vraie dans les cas euclidien et hyperbolique. Dans ce deuxième article, nous passons de l'infinitésimal au local, et montrons que l'espace des polyèdres hyperboliques convexes de combinatoire fixée est localement paramétré par la donnée des angles dièdres, ainsi qu'un résultat similaire pour les cônes-variétés hyperboliques.

\end{minipage}

\normalsize

\section{Introduction}

In his 1968 article \cite{Stoker}, Stoker asks the following question: if $\mathcal{P}$ is a convex polyhedron, then is it true that the internal angles of its faces are determined by the dihedral angles of its edges? This conjecture, originally intended for Euclidean polyhedra, has been readily extended to convex polyhedra in the $3$-sphere or the $3$-dimensional hyperbolic space. In both latter cases, the question becomes whether a spherical or hyperbolic convex polyhedron is determined by its combinatorial type and its dihedral angles. 
In a first article \cite{Maz-GM}, an infinitesimal version of the Stoker problem was proven in the Euclidean and hyperbolic case. It states that there is no nontrivial deformation of a convex hyperbolic polyhedron for which the infinitesimal variation of all dihedral angles vanishes; for a convex polyhedron in Euclidean space, such first-order deformations exist but preserve the internal angles of the faces.

This theorem is actually formulated in the more general setting of cone-manifolds. A hyperbolic or Euclidean cone-$3$-manifold is a constant curvature stratified space, which can be locally described as a gluing of (hyperbolic or Euclidean) tetrahedra. The metric is smooth everywhere except on the singular locus, consisting of glued edges and vertices. Near a singular edge, the metric looks asymptotically like the product of an interval with a $2$-dimensional cone, allowing to define the cone angle of this edge (a more precise definition is given in Section \ref{subsec:hypconeman}). The relationship with polyhedra is straightforward : given a polyhedron $\mathcal{P}$, one can construct its double, which has a natural cone-$3$-manifold structure. If $\mathcal{P}$ is convex, then the cone angles of its double are smaller than  $2\pi$; this restriction will always be present in this article. 

\begin{Thm}[The Infinitesimal Stoker Conjecture for Cone-manifolds, \cite{Maz-GM}]\label{thm:mazgm}\mbox{}\\
Let $\bar M$ be a closed, orientable three-dimensional cone-manifold with all cone angles smaller than $2\pi$. If $\bar M$ is hyperbolic, then 
$\bar M$ is infinitesimally rigid relative to its cone angles, i.e.\ every angle-preserving infinitesimal deformation is trivial.
If $\bar M$ is Euclidean, then every angle-preserving deformation also preserves the spherical links of the codimension $3$ singular points of $\bar M$.

In particular, convex hyperbolic polyhedra are infinitesimally rigid relatively to their dihedral angles, while every 
dihedral angle preserving infinitesimal deformation of a convex Euclidean polyhedron also preserves the internal angles 
of the faces.
\end{Thm}

The goal of this article is to show a local rigidity result in the hyperbolic case, for convex polyhedra and closed cone-manifolds (with cone angles smaller than $2\pi$). The fact that infinitesimal rigidity implies local rigidity in this setting was already proven by Hodgson and Kerckhoff in \cite{HK} and by Wei\ss{} in \cite{Weiss2}, respectively in the case where the singular locus is a link and in the case where the cone angles are smaller than $\pi$ (in both papers, the authors also prove the infinitesimal rigidity). Actually, the technique in Section $6$ of Wei\ss{}'s article is valid as long as the singular locus is a trivalent graph. The main difficulty encountered in this paper is that for a non-trivalent singular locus, ``splitting'' of vertices may occur: consider for instance the double of the polyhedron depicted in Fig.~\ref{fig1}. But note that the topology of the singular locus changes. 

\begin{figure}[h]
\centering
\includegraphics{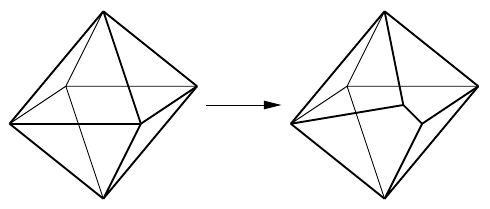}
\caption{Splitting of a vertex}\label{fig1}
\end{figure}

We will show in this article how to circumvent this difficulty and prove a strong version of the local Stoker problem in the hyperbolic case:

\begin{Thm}
The space of convex hyperbolic polyhedra with given combinatorial type is locally parametrized by the tuple of dihedral angles.
\end{Thm}

This main result is consequence of the following parametrization theorem (compare with \cite[Theorem 4.7]{HK} and \cite[Corollary 1.3]{Weiss2}):

\begin{Thm}
Let $\bar M$ be a closed, orientable, hyperbolic cone-$3$-manifold with singular locus $\Sigma$ and whose cone angles are smaller than $2\pi$. Then the space of hyperbolic cone-$3$-manifold structures with singular locus $\Sigma$ near $\bar M$ is locally parametrized by the tuple of cone angles.
\end{Thm}

This result leaves some questions open. Firstly, it does not imply global rigidity, namely, the congruence of two convex hyperbolic polyhedra having same combinatorial type and dihedral angles. It is well known when the dihedral angles are non-obtuse, i.e.~smaller than $\pi/2$: this is the famous Andreev's Theorem \cite{Andreev}, see also \cite{RHD} for a corrected and allegedly more readable proof. But it remains unsolved otherwise, not to mention the case where some dihedral angles are bigger than $\pi$. The same situation happens for hyperbolic cone-$3$-manifolds, where global rigidity is known to hold only in the case where all cone angles are smaller than $\pi$, see \cite{Kojima,Weiss3}. When some cone angles are bigger than $2\pi$, the situation is more contrasted: flexible cone-manifolds exist \cite{Izmestiev}, even if local rigidity still hold in some cases \cite{HK3}.

Secondly, an analogous local result for Euclidean convex polyhedra or cone-manifolds would be very interesting. But the technique used here can certainly not be applied directly, since in the Euclidean case angle-preserving infinitesimal deformations do exist; actually, even what the statement should be is not quite obvious. On the other hand, the Stoker problem is known to be false in the spherical case, see \cite{SchlenkerPoly}.

Finally, to complete the local description of the space of hyperbolic cone-manifolds one has to understand the aforementioned splittings of vertices. This theory is rather delicate and is the subject of a subsequent paper \cite{MontcouqW} in collaboration with Hartmut Wei\ss{}. The main difficulty is that there are infinitely many ways to split a vertex; in particular, and contrary to the present setting, it is no longer the case that the deformation of the double of a polyhedron stays in that class. 

\bigskip

The outline of this article is as follows. In Section 2, we begin by giving a precise definition of a hyperbolic cone-$3$-manifold, then we review classic material about deformations of hyperbolic structure and see what it means for cone-manifolds. In particular, we reformulate the infinitesimal rigidity theorem of \cite{Maz-GM} in this formalism. Roughly speaking, a hyperbolic structure on a $3$-manifold $M$ is locally determined by the conjugacy class of its holonomy representation, i.e.~a homomorphism $\rho : \pi_1(M) \to PSL(2,\C)$. The space of infinitesimal (i.e.~first-order) deformations of this structure is thus identified with the group cohomology space $H^1(\pi_1(M);Ad\circ \rho)$, which in turn can be identified with the cohomology group $H^1(M;\E)$ for $1$-forms with value in a geometric vector bundle $\E$. The infinitesimal rigidity in this setting corresponds to vanishing results for these cohomology spaces. 

The fact that $M$ is the regular part of a cone-manifold with given singular locus has important consequences on its holonomy representation $\rho$, which are best seen on the induced representation $i^* \rho$ on the boundary of a tubular neighborhood of the singular locus. In Section 3, we study this kind of surface group representations and their deformations. As a result, we obtain on the space of representations a system of local coordinates near $i^*\rho$. In Section 4, we show how to lift these local coordinates to the space of hyperbolic structures near $M$; this yields the wanted parametrization on the subset of cone-manifolds with given singular locus. The application to the Stoker problem is then an easy corollary.

\bigskip

It should be noted that results similar to those presented in this article and in \cite{Maz-GM} have been simultaneously and independently found by H.~Wei\ss{} \cite{Weiss4}.

\section{Deformation theory of hyperbolic cone-manifolds and infinitesimal rigidity}
\label{sec:defotheory}

\subsection{Hyperbolic cone-$3$-manifolds}
\label{subsec:hypconeman}

There exist several ways to define cone-manifolds, depending on which aspect the author wants to emphasize. In this article, as we do not need a whole general theory, we will use somewhat simplified definitions, following Thurston \cite{ThurstonShapes}. The interested reader can refer to \cite{Maz-GM} for a more detailed approach.

A spherical cone-surface $\bar S$ is a $2$-dimensional singular Riemannian manifold, such that each point admits a neighborhood in which the expression of the metric $g$ in polar coordinates is 
\begin{equation}\label{eq:dim2}
g= dr^2 + \sin(r)^2 d\theta^2,\ r\in[0,\epsilon),\ \theta \in \R/\alpha\Z.\end{equation}
If $\alpha$ is equal to $2\pi$, then this neighborhood is isometric to a disc in the two-sphere $\mathbb{S}^2$: such a point is called regular. Otherwise, the point is singular and the quantity $\alpha$ is called the cone angle of this singular point. The set of regular points, called the regular part of  $\bar S$, is easily seen to be an open dense subset of $\bar S$; the singular points are isolated and form the singular locus of $\bar S$. In this article, we will mainly focus on orientable cone-manifolds with cone angles smaller that $2\pi$; the Gauss-Bonnet formula implies that such a spherical cone-surface is topologically the $2$-sphere $\mathbb{S}^2$. 
When it has only two singular points, the metric $g$ can be expressed globally as
\begin{equation}g=dr^2 + \sin(r)^2 d\theta^2, \ r\in[0,\pi],\ \theta \in \R/\alpha\Z; \label{eq:2points}
\end{equation}
this type of spherical cone-surfaces will be relevant thereafter.

Likewise, a hyperbolic cone-$3$-manifold  is a $3$-dimensional singular Riemannian manifold, such that each point admits a neighborhood in which the expression of the metric $g$ in spherical coordinates is 
\begin{equation} g = dr^2 + \sinh(r)^2 g_{\bar L},\ r\in[0,\epsilon) \label{eqexprsph} \end{equation}
where $g_{\bar L}$ is the singular metric of a spherical cone-surface $\bar L$, called the link of the point; for the sake of simplicity we will require that $\bar L$ is closed, connected and orientable. As in the surface case, if $\bar L$ is the standard $2$-sphere, then this neighborhood is isometric to a ball in the hyperbolic space $\H^3$, and the point is called regular; otherwise it is called singular. Now if $\bar L$ is a topological $2$-sphere with only two cone points as in \eqref{eq:2points}, the metric $g$ can also be expressed locally in cylindrical coordinates as 
\begin{equation} g = d\rho^2 + \sinh(\rho)^2 d\theta^2 + \cosh(\rho)^2 dz^2,\ \rho\in[0,\epsilon),\ \theta \in \R/\alpha\Z ,\ z \in (-\epsilon,\epsilon) \label{eqexprcyl} \end{equation}
In this coordinate system, the points of the set $\{\rho=0\}$ are singular and share the same local expression of the metric. Their union is called a singular edge and the quantity $\alpha$ is called the cone angle (or dihedral angle) of this singular edge. The union of the singular edges forms an open, dense, $1$-dimensional subset of the singular locus. The remaining singular points are isolated and are called singular vertices; topologically, the singular locus is a graph, geodesically embedded in the cone-manifold.  

\medskip

Given a hyperbolic polyhedron $\mathcal{P}$ (without ``removable edges'', i.e.~edges with dihedral angles equal to $\pi$), we can construct its double by gluing together $\mathcal{P}$ and its mirror image along matching faces. This double is precisely a hyperbolic cone-$3$-manifold, whose singular locus corresponds to the edges and vertices of $\mathcal{P}$, and whose cone angles are exactly twice the dihedral angles of $\mathcal{P}$. As mentioned in the introduction, this construction allows to translate statements relative to cone-manifolds into statements relative to polyhedra; note that if $\mathcal{P}$ is convex, then its dihedral angles are smaller than $\pi$, thus the cone angles of its double are smaller than $2\pi$.

\subsection{Holonomy representation and developing map}

Let $\bar M$ be a connected, orientable hyperbolic cone-$3$-manifold and denote by $M$ (resp. $\Sigma$) its regular part (resp. singular locus). Then $M$ is an incomplete hyperbolic $3$-manifold, whose metric completion is exactly $\bar M$, and we can apply to $M$ the classic machinery of geometric structures (see \cite{Goldman4} for a thorough exposition of the subject, and also \cite{HK} for the case of hyperbolic cone-$3$-manifolds).

Let $\tilde M$ denote the universal cover of $M$ and $\pi : \tilde M \to M$ the associated projection; the hyperbolic metric $g$ on $M$ lifts to a hyperbolic metric $\tilde g = \pi^*g$ on $\tilde M$, and $\pi$ becomes a local isometry. Let us choose a base-point $\tilde x$ in $\pi^{-1}(x)$. Then the action of $\pi_1(M,x)$ on $\tilde M$ via deck transformations is well-defined; it is transitive on the fibers. The hyperbolic metric $\tilde g$ on $\tilde M$ allows to define by analytical continuation the developing map $dev :  \tilde{M} \to \H^3,$
which is a local isometry, well-defined up to an isometry of $\H^3$ (acting by left composition). In particular $\tilde g = dev^* g_{\H^3}$, and the metric on $M$ is completely determined by the developing map and the projection $\pi$. The developing map clearly features an equivariant property: there exists an application $hol : \pi_1(M,x) \to \Isom^+(\H^3)$, called the holonomy representation, such that for all $p \in \tilde M$ and $\gamma \in \pi_1(M,x)$,
$$dev(\gamma.p) = hol(\gamma).dev(p).$$
The holonomy representation is well-defined up to conjugation by an isometry of $\H^3$. Note that contrary to the complete case, it has no reason to be faithful nor discrete.

Let us denote by $R(\pi_1(M,x),\Isom^+(\H^3))$ the representation space, i.e.~the set of all group homomorphisms from $\pi_1(M,x)$ to $\Isom^+(\H^3)$, endowed with the compact-open (or pointwise convergence) topology. Denote also by $X(\pi_1(M,x),\Isom^+(\H^3))$ the $\Isom^+(\H^3)$-character variety of $\pi_1(M,x)$, i.e.~the quotient of the representation space by the action of $\Isom^+(\H^3)$ by conjugation. The holonomy representation being determined, up to conjugation, by the hyperbolic metric $g$ on $M$, we get a map 
\begin{equation} \label{eqmaphol1} \{\mbox{hyperbolic metrics on } M\} \to X(\pi_1(M,x),\Isom^+(\H^3)).\end{equation}

Now if $M$ is the regular part of an oriented cone-manifold, its holonomy representation has some additional properties. 
If $p$ is a singular point of $\bar M$, denote by $L_p$ the regular part of its spherical link and by $N_p$ the regular part of small enough neighborhood of $p$; then $L_p$ is a deformation retract of $N_p$. The inclusion map $i : N_p \hookrightarrow M$ induces a representation $i^*hol : \pi_1(N_p)\simeq \pi_1(L_p) \to \Isom^+(\H^3)$. The image of $i^*hol$ is then contained in a maximal compact subgroup $K$ of $\Isom^+(\H^3)$:  the induced representation is indeed (via the identification $K\simeq SO(3)$) the holonomy representation of the spherical structure on $L_p$. This observation is of course most relevant when $p$ is a singular vertex of $\Sigma$. When applied to a point in a singular edge $e$, it shows that if $\gamma \in \pi_1(M,x)$ is freely homotopic to a meridian around $e$, then $hol(\gamma)$ is an elliptic isometry, whose rotation angle is equal modulo $2\pi$ to the cone angle of this singular edge.

\subsection{Local deformation of a hyperbolic structure}

Let $g_1$ and $g_2$ be two incomplete hyperbolic metrics on an orientable $3$-manifold $M$, whose metric completions $\bar M_1$ and $\bar M_2$ are cone-manifolds. We will say that these two hyperbolic cone-manifolds are equivalent if there exists a diffeomorphism $\phi$ of $M$, isotopic to the identity, such that $g_1 = \phi^* g_2$. A cone-manifold structure is then an equivalence class for this relation. More generally, we can define the same equivalence relation for any hyperbolic metrics on $M$; we will denote the quotient space (or hyperbolic structure space) by $\mathcal{M}$. Since two equivalent metrics induce the same (up to conjugation) holonomy representation, we can quotient \eqref{eqmaphol1} as a map
\begin{equation} \label{eqmaphol2} \mathcal{M} \to X(\pi_1(M,x),\Isom^+(\H^3)).\end{equation}

There is another equivalence relation we need to introduce, namely the one induced by thickening. For simplicity, we will formulate it when $M$ is diffeomorphic to the interior of a compact manifold with boundary, which is always the case if $M$ is the regular part of a cone-manifold. Then $M$ is diffeomorphic to $M \amalg\, \d M\! \times\! [0,\epsilon)$. Given a hyperbolic metric $g$ on $M$, this diffeomorphism pulls it to a metric on $M \amalg\, \d M\! \times\! [0,\epsilon)$; let us denote its restriction to $M$ by $g'$. The metric $g$ is called a thickening of $g'$, and both give rise to the same (up to conjugation) holonomy representation. Let us denote by $\sim$ the induced equivalence relation on $\cal M$; we obtain a natural map 
\begin{equation} \label{eqmaphol3} [hol] : \mathcal{M}_{/\sim} \to X(\pi_1(M,x),\Isom^+(\H^3)).\end{equation}
We would like this map to be a homeomorphism, but actually this is not the case unless we restrict it to irreducible representations. More precisely, we have the following result, which is true in a much more general framework: 

\begin{Thm}[Deformation Theorem, see Section 3 of \cite{Goldman4}] \label{thm:Goldman}\mbox{}\\
Let $\mathcal{R} = [hol]^{-1}(X^{irr}(\pi_1(M,x),\Isom^+(\H^3))) \subset \mathcal{M}_{/\sim}$. Then the map $$[hol] : \mathcal{R} \to X^{irr}(\pi_1(M),\Isom^+(\H^3))$$ is a local homeomorphism.
\end{Thm}

This important result shows that in order to study the local deformation of a hyperbolic structure, it is sufficient to understand the local structure of the irreducible part of the character variety $X(\pi_1(M,x),\Isom^+(\H^3))$.

This space can actually be studied in a quite general setting. For any discrete group $\Gamma$ and any Lie group $G$, we can consider the set  $R(\Gamma,G)$ of all homomorphisms from $\Gamma$ to $G$; such an homomorphism is called a $G$-representation of $\Gamma$. If $\Gamma$ admits a finite presentation $\<s_1,s_2,\ldots,s_n\ |\ f_1(s_1,\ldots,s_n),\ldots,f_p(s_1,\ldots,s_n)\>$ then its representation space $R(\Gamma,G)$ can be identified with the subset 
$$\{(x_1,\ldots,x_n) \in G^n\ ;\ f_i(x_1,\ldots,x_n)=e\ \forall i=1\ldots p \},$$ which is an algebraic (resp.~analytic) variety as soon as $G$ is an algebraic group (resp.~complex Lie group). This identification allows to define the usual topology as well as the Zariski one on $R(\Gamma,G)$, the former being the pointwise convergence one. The group $G$ acts on $R(\Gamma,G)$ by conjugation and we denote the quotient space by $X(\Gamma,G)$. Note that this quotient is not well-behaved in general: it is usually non-Hausdorff, having non-closed points. In the literature, the notation $X(\Gamma,G)$ generally stands for the algebraic-geometric quotient, whose elements correspond to the closure of conjugation classes in $R(\Gamma,G)$; since we will be dealing with irreducible representations, the difference between with these two quotients will not be an issue here.

Let $\rho : \Gamma \to G$ be a representation. A first-order (or infinitesimal) deformation of $\rho$ is then a function $\dot \rho : \Gamma \to \mathfrak{g}$, where $\mathfrak g$ is the Lie algebra of $G$, satisfying the cocycle condition
\begin{equation}\dot \rho(\gamma_1 \gamma_2) = \dot \rho(\gamma_1) + Ad(\rho(\gamma_1))(\dot \rho(\gamma_2)) \ \ \forall \gamma_1, \gamma_2 \in \Gamma \label{eqcocycle} \end{equation}
We denote by $Z^1(\Gamma,Ad\circ \rho)$ the space of all maps  from $\Gamma$ to $\mathfrak{g}$ satisfying this cocycle condition; it is canonically identified with the Zariski tangent space of $R(\Gamma,G)$ at $\rho$. 

A deformation of the representation $\rho$ is called trivial if it corresponds to a conjugation, namely if it is of the form $\dot \rho(\gamma) = \frac{d}{dt}|_{_{t=0}}\, g_t \rho(\gamma) g_t^{-1}$. This is equivalent to satisfying the coboundary condition: there exists $v \in \mathfrak{g}$ such that for all $\gamma \in \Gamma$, 
$$\dot \rho(\gamma) = v - Ad(\rho(\gamma))(v).$$ 
We denote by $B^1(\Gamma,Ad\circ \rho)$ the space of all maps  from $\Gamma$ to $\mathfrak{g}$ satisfying this coboundary condition. The first group cohomology space $H^1(\Gamma,Ad\circ \rho)$ is then defined as the quotient $Z^1(\Gamma,Ad\circ \rho) / B^1(\Gamma,Ad\circ \rho)$. It is canonically identified with the Zariski tangent space of $X(\Gamma,G)$ at the equivalence class of $\rho$. 

If $\Gamma$ is the fundamental group of a connected manifold $M$, we can relate this construction to more usual cohomology spaces. More precisely, we can define on $M$ a vector bundle $\E$ associated to $\rho$ as follows. On the universal cover $\tilde M$ of $M$ we can consider the trivial bundle $\tilde \E = \tilde M \times \mathfrak{g}$; it has a trivial flat connexion $D$. Recall that $\pi_1(M,x)$ acts by deck transformations on $\tilde M$, transitively on the fibers. Then we can define $\E$ as the quotient $(\tilde M \times \mathfrak{g}) / \sim$, where $(p,v) \sim (\gamma.p, Ad \circ \rho(\gamma) (v))$. The connexion on $\tilde \E$ descends to a flat connexion, still denoted $D$, on $\E$. 
The reason for introducing $\E$ is because there exists a classic isomorphism $$H^1(M;\E) \stackrel{\sim}{\to} H^1(\pi_1(M,x);Ad\circ \rho)$$ between $\E$-valued form cohomology and group cohomology; this isomorphism is given by integration of a closed $\E$-valued $1$-form along loops in $M$. This allows to make use of geometric analysis techniques in the study of local and infinitesimal deformations.

 For an orientable hyperbolic $3$-manifold $M$, $\E$ is a $\mathfrak{sl}(2,\C)$-bundle, and can be interpreted as the bundle of infinitesimal local isometries, or local Killing vector fields. The fiber over a point $x \in M$ corresponds to the vector space of germs at $x$ of Killing vector fields, and flat sections of $\E$ over an open set $U$ corresponds to Killing vector fields on $U$. The vector bundle $\E$ admits a geometric decomposition $\mathcal{P} \oplus \mathcal{K}$, where at a point $x$ the fiber $\mathcal{P}_x$ corresponds to ``infinitesimal pure translations'' through $x$ (i.e.~derivatives of, or Killing fields integrating as, hyperbolic isometries whose axis goes through $x$) and the fiber $\mathcal{K}_x$ to infinitesimal rotations (i.e.~elliptic isometries) fixing $x$; the sub-bundle $\mathcal{P}$ is naturally identified with the tangent bundle $TM$. Note that the flat connexion on $\E$ does not preserve this decomposition. 
 
 An element of $H^1(M;\E)$ will be called an infinitesimal deformation of the hyperbolic structure on $M$ (or of $M$ for short). Indeed, $H^1(M;\E)$ is identified to the Zariski tangent space of $X(\pi_1(M),\Isom^+(\H^3))$ at $[\rho]$. If $[\rho]$ is a smooth point of this quotient representation space, then this is the usual tangent space, and if $\rho$ is irreducible then by Theorem \ref{thm:Goldman} it corresponds to the tangent space to the set of hyperbolic structure on $M$. So elements of $H^1(M;\E)$ are in bijection with first-order deformations of the hyperbolic structure on~$M$.

\subsection{Infinitesimal rigidity}
\label{subsec:infrig}

We recall the following result from \cite{Maz-GM}:

\begin{Thm}
Let $\bar M$ be a closed, orientable, hyperbolic cone-$3$-manifold with all cone angles smaller than $2\pi$. Then $\bar M$ is infinitesimally rigid relative to its cone angles, i.e.\ every angle-preserving infinitesimal deformation is trivial.
\end{Thm}

Although it is not explicitly stated, this theorem only deals with infinitesimal deformations that preserve the singular locus, i.e.~no splitting of vertex occurs. If it were the case, new singular edges would appear; the more general rigidity result of \cite{MontcouqW} takes their lengths into account.
 
 The above theorem was actually proven in the language of curvature-preserving infinitesimal deformations of the metric tensor.
The relation with the formalism we have exposed earlier, i.e.~deformations as closed $\E$-valued $1$-forms, is very well described in \cite{HK}; for the sake of clarity we will outline this relation now. If $\omega$ is a closed $\E$-valued $1$-form, we can lift it to a closed $\tilde \E$-valued $1$-form $\tilde \omega$ on $\tilde M$. Since $\tilde M$ is simply connected, $\tilde \omega$ is exact and hence equal to $D\tilde s$ for some section $\tilde s$ of $\tilde \E$. Similarly to $\E$, the fiber bundle $\tilde \E$ decomposes as $\tilde{\mathcal{P}} \oplus \tilde{\mathcal{K}}$,
 where $\tilde{\mathcal{P}} \simeq T \tilde M$. The $\tilde{\mathcal{P}}$-part of $\tilde s$ can thus be identified with a vector field $\tilde X$ on $\tilde M$, which corresponds to the deformation of the developing map. In general $\tilde X$ does not descend to a vector field on $M$, but it satisfies an equivariant property (it is ``automorphic'' in the language of \cite{HK}). As a consequence the infinitesimal deformation of the metric tensor $\tilde h = L_{\tilde X} \tilde g$ actually descends to a well-defined curvature-preserving deformation $h$ on $M$. Alternatively, we can also consider the $\mathcal{P}$-part of $\omega$; it identifies with a $TM$-valued $1$-form (but not closed in general). Using the isomorphism between $TM$ and $T^*M$ given by the metric $g$, it can be seen as a section of $T^*M\otimes  T^*M$. Its symmetric part $h$ is then the infinitesimal deformation of the metric.

On the other hand, if $h$ is a curvature-preserving infinitesimal deformation of $g$, then it is locally an infinitesimal isometry, i.e.~it can be written locally as $L_X g$ for some local vector field $X$. The lift $\tilde h$ of $h$ to the universal cover $\tilde M$ is then equal to $L_{\tilde X} \tilde g$ for some globally defined vector field $\tilde X$, which is automorphic. Now there exists a section $\tilde r$ of $\tilde{\mathcal{K}}$ such that the differential $\tilde \omega = D\tilde s$ of the section $\tilde{s} = \tilde{X} + \tilde{r}$ descends to a (closed) $\E$-valued $1$-form $\omega$ on $M$, which is the desired deformation. Such a section $\tilde r$ can be found by ``osculating'' $\tilde X$.

\bigskip

For a hyperbolic cone-manifold $\bar M$ with singular locus $\Sigma$, let $\bar U_\epsilon$ be a small enough tubular neighborhood of $\Sigma$. We will denote by $M_\epsilon$ its complement and by $U_\epsilon = \bar U_\epsilon \setminus \Sigma$ its regular part. Then $M_\epsilon$ is a manifold with boundary whose interior is diffeomorphic to $M$. We will denote its boundary by $\Sigma_\epsilon$; it is a deformation retract of $U_\epsilon$. These notations will be used extensively in the remainder of this article. 

It is clear that the hyperbolic cone-structure on $\bar U_\epsilon$ is relatively constrained. In particular, any infinitesimal deformation $h$ of the cone metric $g$ that preserves the hyperbolic cone-manifold structure and the topology of the singular locus can be put in ``standard form'' in $U_\epsilon$, see \cite[Section 3]{Maz-GM}. More precisely, near $\Sigma$ and up to an infinitesimal isometry (i.e.~up to a $2$-tensor of the form $L_X g$ for some global vector field $X$), the infinitesimal deformation is a linear combination of the following elementary deformations.

\begin{enumerate}
\item Infinitesimal deformations modifying the length, resp.~the twist parameter of a singular edge: in the cylindrical coordinates of \eqref{eqexprcyl}, they are of the form $h^e_\lambda = \cosh(\rho)^2 \phi'(z) dz^2$, resp.~$h^e_\tau = \sinh(\rho)^2 \phi'(z) d\theta dz$, where $\phi(z)$ is a smooth nondecreasing function vanishing for $z \leq -\epsilon/2$ and equal to $1$ for $z \geq \epsilon/2$. The corresponding $\mathcal E$-valued $1$-forms are simple to compute. Since the vector fields $\frac{\d}{\d z}$ and $\frac{\d}{\d\theta}$ are local Killing vector fields, they correspond to local parallel sections of $\E$, denoted $\sigma_{\d / \d z}$ and $\sigma_{\d / \d\theta}$. With these notations, the deformations $h^e_\lambda$ and $h^e_\tau$ correspond to the following closed $\E$-valued $1$-forms on $U_\epsilon$:
$$ \omega^e_\lambda = d\phi \wedge \sigma_{\d / \d z}, \qquad \omega^e_\tau = d\phi \wedge \sigma_{\d / \d\theta}$$

\item Infinitesimal deformations modifying the link $\bar L$ of a singular vertex $p \in \Sigma$: in the spherical coordinates of \eqref{eqexprsph}, they are of the form $h_p = \sinh(r)^2 h_{\bar L}$, where $h_{\bar L}$ is a curvature-preserving infinitesimal deformation of $\bar L$. To compute an associated $\E$-valued $1$-form, we first observe that in a neighborhood of $p$, the bundle $\E$ splits as $\E_1 \oplus \E_2$, where $\E_1$, resp.~$\E_2$, is the subbundle of infinitesimal rotations fixing $p$, resp.~infinitesimal translations through $p$. If we identify the level sets $\{r=constant\}$ with $\bar L$, then $\E_1$ corresponds to the bundle of infinitesimal spherical isometries of $\bar L$, so that we can associate to $h_p$ a closed $\E_1$-valued $1$-form $\omega_p$, independent of $r$. To be more precise we can distinguish two cases.
\begin{enumerate}
\item The deformation preserves the cone angles of $\bar L$. Then we can always assume that $h_{\bar L}$ is supported away from the cone points of the link, so that $h_p$ and $\omega_p$ vanish near the singular edges.
\item The deformation modifies some cone angles of $\bar L$, and hence also the cone angles of the corresponding singular edges of $\Sigma$ and the links of the facing vertices. Then we can always assume that in the coordinates of \eqref{eq:dim2} near a cone point, $h_{\bar L}$ is a multiple of $\sin(r)^2d\theta^2$, so that in the cylindrical coordinates of \eqref{eqexprcyl} $h_p$ is a multiple of $\sinh(\rho)^2d\theta^2$ and $\omega_p$ a multiple of $d\theta \wedge \sigma_{\partial/\partial \theta}$.
\end{enumerate}
\end{enumerate}

As mentioned above, an infinitesimal deformation is in standard form if it is equal in a neighborhood of $\Sigma$ to a linear combination of the above model deformations. With this terminology, we can restate the main infinitesimal rigidity result:

\begin{Thm}[Infinitesimal rigidity, \cite{Maz-GM}] \label{thm:rigL2} Let $\bar M$ be a closed, orientable, hyperbolic cone-$3$-manifold with all cone angles smaller than $2\pi$. Then any curvature-preserving infinitesimal deformation $h$ of $\bar M$ in standard form that does not modify the cone angles is trivial, i.e.~is equal to $L_Xg$ for some globally defined vector field $X$.
\end{Thm}

This theorem is proved using geometric analysis, and relies on the study of the second-order partial differential equation satisfied by $h$. This infinitesimal rigidity result has important consequences in terms of $\E$-valued cohomology groups: 

\begin{Prop}\label{prop:H1}
Under the assumptions of Theorem \ref{thm:rigL2}, the map $H^1(M,U_\epsilon;\E) \to H^1(M;\E)$ is the zero map, and the map $H^1(M;\E) \to H^1(U_\epsilon;\E)$ is injective with half-dimensional image.
\end{Prop}

\begin{proof}
Let $\omega \in \Omega^1(M,\E)$ be a closed $\E$-valued $1$-form, equal to zero over $U_\epsilon$. The corresponding infinitesimal deformation $h$ of the metric tensor is also zero over $U_\epsilon$ and therefore is in standard form and preserves the cone angles. We can then apply Theorem \ref{thm:rigL2}, which shows that $h$ is trivial; consequently $\omega$ is a coboundary. This proves the first part of the proposition, namely the vanishing of the map $H^1(M,U_\epsilon;\E) \to H^1(M;\E)$

The fact that this implies the second part has been already proven in \cite{HK,Weiss2}; here is how it goes. By Poincar\'e duality, we obtain that the dual map $H^2(M,U_\epsilon;\E) \to H^2(M;\E)$ also vanishes. Now we can look at (part of) the long exact sequence of the pair $(M,U_\epsilon)$:
$$H^1(M,U_\epsilon;\E) \mathop{\longrightarrow}^1 H^1(M;\E) \mathop{\longrightarrow}^2 H^1(U_\epsilon;\E) \mathop{\longrightarrow}^3 H^2(M,U_\epsilon;\E) \mathop{\longrightarrow}^4 H^2(M;\E)$$
The maps $(1)$ and $(4)$ vanish, so $(2)$ is injective and $(3)$ is surjective. This implies that $\dim H^1(U_\epsilon;\E) = \dim H^1(M;\E) + \dim H^2(M,U_\epsilon;\E)$, but the two dimensions at the right hand side are equal since by Poincar\'e duality $H^2(M,U_\epsilon;\E) \simeq (H^1(M;\E))^*$.
\end{proof}

Let us denote by $\rho$ the holonomy representation of $M$. 
We have seen that $H^1(M;\E)$ can be identified with the tangent space of $X(\pi_1(M),\Isom^+(\H^3))$ at $[\rho]$. And since the inclusion of $\Sigma_\epsilon$ into $U_\epsilon$ is a deformation retract, $H^1(U_\epsilon;\E)$ is equal to $H^1(\Sigma_\epsilon;\E)$, which in turn can be identified with the tangent space of the character variety of $\Sigma_\epsilon$. More precisely, let $\mathcal C$ be the set of the connected components of the singular locus and for each $c\in \mathcal C$, let $\Sigma^c_\epsilon$ the corresponding component of $\Sigma_\epsilon$; let also $i_c$ be the inclusion map $\Sigma_\epsilon^c \hookrightarrow M$. Then $H^1(\Sigma_\epsilon;\E)$ is the tangent space of $\bigtimes_{c\in \mathcal C} X(\pi_1(\Sigma_\epsilon^c),\Isom^+(\H^3))$ at the conjugacy classes of the representations $i_c^* \rho$ induced by $\rho$ on each of the $\Sigma_\epsilon^c$.
With these identifications, the above map $H^1(M;\E) \to H^1(\Sigma_\epsilon;\E)$ is exactly the tangent map at $[\rho]$ of the restriction map $r : X(\pi_1(M),\Isom^+(\H^3)) \to \bigtimes_{c\in \mathcal C} X(\pi_1(\Sigma_\epsilon^c),\Isom^+(\H^3))$.

Recall that according to Theorem \ref{thm:Goldman}, we want to study locally $X(\pi_1(M),\Isom^+(\H^3))$; so we will begin by looking at $X(\pi_1(\Sigma_\epsilon^c),\Isom^+(\H^3))$ and then use the above proposition to lift the results to $X(\pi_1(M),\Isom^+(\H^3))$.

\section{Representations of surface groups}

Our eventual goal is to find a local system of coordinates on $X(\pi_1(M),\Isom^+(\H^3))$ which will yield the local parametrization by the cone angles on the space of cone-manifolds. In light of the local properties of the restriction map $r$ from $X(\pi_1(M),\Isom^+(\H^3))$ to $\bigtimes_{c\in \mathcal C} X(\pi_1(\Sigma_\epsilon^c),\Isom^+(\H^3))$, we start in this section by studying the character variety of $\Sigma_\epsilon$ and showing that we can find functions with nice geometric interpretations on it (Theorem \ref{thm:F}). In Section \ref{sec:locdef} we will see that these functions indeed lift to the system of coordinates we are looking for on the character variety of $M$. Since we want to determine when the restriction of a representation to a vertex's link has values in a compact subgroup, a first step is the existence of convenient coordinate charts on the character varieties of the links (Proposition \ref{propfunctions}); Theorem \ref{thm:F} follows by applying a gluing process. But before that we need some facts about our representation varieties.

It is well known that the group $\Isom^+(\H^3)$ of direct isometries of the hyperbolic $3$-space is isomorphic to the Lie group $PSL(2,\C)$. Instead of considering $PSL(2,\C)$-representations and characters, we have found it somewhat easier to work with the universal cover $SL(2,\C)$, and this is possible thanks to the following result of Culler:

\begin{Thm}[Representations Lifting, \cite{Culler}]\mbox{}\\
Let $M$ be an orientable, not necessarily complete, hyperbolic $3$-manifold and let $\rho :  \pi_1(M) \to PSL(2,\C)$ be its holonomy representation. Then $\rho$ can be lifted to $SL(2,\C)$, i.e.~there exists $\tilde \rho :  \pi_1(M) \to SL(2,\C)$ such that $\rho = \pi \circ \tilde \rho$ where $\pi$ is the projection $SL(2,\C)\to PSL(2,\C)$.
\end{Thm}

In the following, we will only work with representations into $SL(2,\C)$; since the projection map from $X(\Gamma,SL(2(\C))$ to $X(\Gamma,PSL(2,\C))$ is a local diffeomorphism, this will cause no difference.
This result is also true for a (possibly incomplete) spherical surface: its holonomy representation can be lifted from $SO(3)$ to $SU(2)$. So here as well we will only work with representations into $SU(2)$. Now the inclusion map $SU(2) \hookrightarrow SL(2,\C)$ induces an injective map $R(\Gamma,SU(2)) \to R(\Gamma, SL(2,\C))$. Two unitary representations are conjugated if and only if they are unitarily conjugated; in other words, there exists a well-defined injective map $X(\Gamma, SU(2)) \to X(\Gamma, SL(2,\C))$. Consequently we will often think of these two spaces as lying one inside the other. The following result shows that the representations that we will consider in this article are irreducible:

\begin{Thm}\label{thm:irred}
Let $\bar S$ be a closed, orientable spherical cone-surface with at least three singular points and cone angles smaller than $2\pi$, and let $\rho : \pi_1(S) \to SU(2)$ be (the lift of) its holonomy representation. Then $\rho$ is irreducible. In particular, its conjugacy class $[\rho]$ is a smooth point of $X(\pi_1(S),SU(2))$ and of $X(\pi_1(S),SL(2,\C))$
\end{Thm}

\begin{proof}
Since all cone angles are smaller than $2\pi$, by the Gauss-Bonnet formula the regular part $S$ of $\bar S$ is homeomorphic to a $d$-punctured sphere ($d\geq 3$); consequently its fundamental group is a free group on $d-1$ generators, generated by loops going around the singular points. Thus the representation space $R(\pi_1(S),G)$ can be identified with $G^{d-1}$, which is smooth, and it is easy to see that for $G=SU(2)$ or $SL(2,\C)$, the action of $G$ by conjugation on the subset of irreducible representations of $R(\pi_1(S),G)$ is proper. So it is enough to prove that $\rho$ is irreducible.

We begin by noting that the regular part $S$ of $\bar S$, although an incomplete Riemannian manifold, is geodesically connected, i.e.~there exists a (shortest) geodesic between any pair of points of $S$. Indeed, $\bar S$ is a complete compact length space, so there exists a shortest path $\gamma$ in $\bar S$ between any pair of points of $S$. But a path going through a cone point of angle smaller than $2\pi$ cannot be a shortest path, so $\gamma$ is in fact contained in $S$ and is thus a (Riemannian) geodesic. 

Now choose a base point $x$ in $S$ such that the shortest geodesic from $x$ to any cone point is unique, and consider the complement $C$ of the cut locus of $x$. Via the embedding of $C$ in $\tilde S$ and the developing map, we can construct a local isometry $f : C \to \mathbb{S}^2$. We claim that $f$ is injective, and (the closure of) its image is thus a fundamental polygon for $\bar S$. Indeed, suppose $y_1$ and $y_2$ are two points of $C$ such that $f(y_1) = f(y_2)$. For $i=1$ or $2$, by definition of $C$ the shortest geodesic $\gamma_i$ from $x$ to $y_i$ is contained in $C$, so $f(\gamma_1)$ and $f(\gamma_2)$ are two geodesics in $\mathbb{S}^2$ with the same endpoints $f(x)$ and $f(y_1)=f(y_2)$. So either they are equal, and then $\gamma_1 = \gamma_2$ and so $y_1 = y_2$, or one of them, say $f(\gamma_1)$, goes through the antipode of $f(x)$, which is a conjugate point to $f(x)$ along $f(\gamma)$. This latter case implies that $x$ has a conjugate point along $\gamma_1$, which by definition cannot belong to $C$, and this is impossible since $\gamma_1$ is contained in $C$.

Finally, let $p$ be a cone point of $\bar S$. There is a unique shortest geodesic from $x$ to $p$, so $p$ corresponds to a unique point, that we will denote by $f(p)$, in the boundary of $f(C) \subset \mathbb{S}^2$. Then the holonomy of the loop based at $x$ and going around $p$ is a rotation centered at $f(p)$, whose angle is equal to the cone angle at $p$. Now since $\pi_1(S,x)$ is generated by loops going around the cone points, it implies that the image of the holonomy representation $\rho$ is generated by rotations centered at the $f(p_i)$, which are distinct points of $\mathbb{S}^2$ because $f$ is injective. And since $\bar S$ has at least three cone points, all these rotations cannot have the same axis, and hence $\rho$ is irreducible.
\end{proof}

If $p$ is a  singular vertex of a hyperbolic cone-$3$-manifold, we can apply this result to the induced holonomy representation on the spherical link of $p$ to obtain the following corollary:

\begin{Cor}\label{cor:rhoirred}
Let $\bar M$ be a hyperbolic cone-$3$-manifold with cone angles smaller than $2\pi$ and such that its singular locus contains at least a singular vertex of valency greater than or  equal to $3$. Then its holonomy representation is irreducible.
\end{Cor}

This statement is actually true for any finite volume hyperbolic cone-$3$-manifolds (see \cite[Lemma 6.35]{Weiss2}), but we will not need this fact here. 

\bigskip

If $\gamma \in \pi_1(S)$ is a loop on a surface $S$, we can define a function $\tr_\gamma : R(\pi_1(S), SL(2,\C)) \to \C$ by setting $\tr_\gamma (\varrho) = \tr (\varrho(\gamma))$. Now if $l$ is a closed curve on $S$, all the loops in $\pi_1(S)$ freely homotopic to $l$ are conjugated to one another, and since the trace is invariant by conjugation we can still define the function $\tr_l$ as being equal to $\tr_\gamma$ for any $\gamma \in \pi_1(S)$ freely homotopic to $l$. The invariance of the trace by conjugation also means that $\tr_\gamma = \tr_l$ descends to a function (keeping the same notation) $\tr_\gamma = \tr_l : X(\pi_1(S), SL(2,\C)) \to \C$. In the exact same way, we can define functions $\tr_\gamma$ and $\tr_l : X(\pi_1(S),SU(2)) \to \R$. These trace functions will give us local coordinates on the character varieties.

\begin{Prop}\label{propfunctions}
Let $\bar S_d$ be an orientable, closed spherical cone-surface with $d$ singular points ($d\geq3$) of cone angles smaller than $2\pi$, and let $S_d$ be its regular part. Let $(\gamma_k)_{1\leq k \leq d}$ be loops around the singular points, such that $\<\gamma_1,\ldots,\gamma_d\ |\ \gamma_1\ldots\gamma_d=1\>$ is a presentation of $\pi_1(S_d)$, and let $\rho : \pi_1(S_d) \to SU(2)\subset SL(2,\C)$ be (the lift of) the holonomy representation. Then there exist local complex-valued functions $f_1,...,f_{2d-6}$  on a neighborhood $U$ of $[\rho]$ in $X(\pi_1(S_d),SL(2,\C))$ such that the family $(\tr_{\gamma_k})_{1\leq k \leq d} \cup (f_k)_{1\leq k \leq 2d-6}$ is a local system of complex coordinates in $U$ and that $X(\pi_1(S_d),SU(2))\cap U = \{ \Im \tr_{\gamma_1}= \ldots =\Im \tr_{\gamma_d} = \Im f_1=\ldots=\Im f_ {2d-6}=0\}$.
\end{Prop}

\begin{proof}
We remark first that by Gauss-Bonnet formula, $S_d$ is a sphere with $d$ holes, which justifies the existence of the given presentation of its fundamental group. The result is then rather standard, but we give a proof for the sake of completeness.
We begin by showing that the functions $(\tr_{\gamma_k})_{1\leq k \leq d}$ have linearly independent derivatives at $T_\rho R(\pi_1(S_d),SL(2,\C))$. Note that since $\gamma_d$ is equal to $(\gamma_1\ldots\gamma_{d-1})^{-1}$ and the representations have values in $SL(2,\C)$, the function $\tr_{\gamma_d}$ is equal to $\tr_{\gamma_1\ldots\gamma_{d-1}}$.
As before, we will identify $R(\pi_1(S_d),SL(2,\C))$ to $SL(2,\C)^{d-1}$ by associating to a representation $\rho$ the $(d-1)$-uple $(m_1,\ldots,m_{d-1})$ where $m_k = \rho(\gamma_k)$ for $k=1\ldots d$. Note also that since $m_k$ is the holonomy around a singular point, it is different from the identity.
With this identification, the functions $\tr_{\gamma_k}$, $k=1\ldots d-1$, are clearly linearly independent in a neighborhood of $\rho$.

Now suppose that there exists a dependence relation between the derivatives at $\rho \simeq (m_1,\ldots,m_{d-1})$, namely 
$d\tr_{\gamma_d} = \sum_{1\leq k\leq d-1} \lambda_k d\tr_{\gamma_k}.$ 
This means that for all $(v_1,\ldots, v_{d-1}) \in \mathfrak{sl}(2,\C)^{d-1}$, we have
$$\sum_{1\leq k\leq d-1}  \tr(m_1\ldots m_{k-1}\, v_k m_k\,  m_{k+1}\ldots m_{d-1}) = \sum_{1\leq k\leq d-1}  \lambda_k \tr(v_k m_k).$$
In particular, $\tr(m_1\ldots m_{k-1} v\, m_k m_{k+1}\ldots m_{d-1}) = \lambda_k \tr(v\, m_k)$ for all $k$, $1\leq k \leq d-1$, and all $v \in \mathfrak{sl}(2,\C)$. This implies that $m_km_{k+1}...m_{d-1} m_1...m_{k-1} - \lambda_k m_k$ is a scalar multiple of the identity, and so that $m_k$ and $m_k m_{k+1} ... m_{d-1} m_1... m_{k-1}$ commute (hence also $m_k$ and $m_{k+1}... m_{d-1} m_1... m_{k-1}$). But this being true for all $k$, it implies that $m_{k+1}... m_{d-1} m_1... m_{k-1} m_k = m_1... m_{d-1} = m_d^{-1}$ for all $k$, and in particular $m_k$ and $m_d$ commute for all $k$, which is impossible since $\rho$ is irreducible. 

We have seen (Theorem \ref{thm:irred}) that the conjugacy class $[\rho]$ is a smooth point of both character varieties $X(\pi_1(S),SU(2))$ and $X(\pi_1(S),SL(2,\C))$; in particular, the tangent space at $[\rho]$ is well-defined. The above trace functions are invariant by conjugation, so they descend to these two quotient spaces, and their derivatives at $[\rho]$ are still linearly independent. The rest of the proof follows from the fact that $X(\pi_1(S_d),SL(2,\C))$ has complex dimension $3(d-1)-3 = 3d-6$ and that $SU(2)$ is a real form of the Lie group $SL(2,\C)$. 
\end{proof}

\bigskip

Let $\rho$ be the holonomy representation of a closed, orientable hyperbolic cone-$3$-manifold $\bar M$, with singular locus $\Sigma$ and cone angles smaller than $2\pi$. As in Section \ref{subsec:infrig}, the boundary of a tubular neighborhood of $\Sigma$ is a surface denoted by $\Sigma_\epsilon$. It is not necessarily connected; its components correspond to those of $\Sigma$. Let $\Sigma^c$ be a component of $\Sigma$ which is not a circle, and denote by $V$, resp. $E$ the set of its singular vertices, resp. edges. Let $\Sigma_\epsilon^c$ be the corresponding component of $\Sigma_\epsilon$; it is a surface of genus $g\geq 2$.

By a slight abuse of notation, $\rho$ will also denote the induced representation on $\Sigma_\epsilon^c$. Let $(\mu_e)_{e \in E}$ be a family of meridians, i.e.simple closed curves on $\Sigma_\epsilon^c$ going around the singular edges; they split $\Sigma_\epsilon^c$ into a family $(S_v)_{v \in V}$ of $d(v)$-punctured spheres, where $d(v)$ is the valency of the vertex $v$. Let us denote by $i_v$ the inclusion $S_v \hookrightarrow \Sigma_\epsilon^c$. We have seen that the induced representation $i_v^*\rho$ on $S_v$ is actually the holonomy representation of the spherical cone-surface structure on the link $L_v$ of $v$; it satisfies the hypotheses of Proposition \ref{propfunctions} where the $\gamma_k$ can be chosen (up to free homotopy) among the $\mu_e$. We will denote by $f^v_k$, $k=1\ldots 2d(v)-6$, the corresponding local functions mentioned in Proposition \ref{propfunctions}; they can be pulled back to functions on $X(\pi_1(\Sigma_\epsilon^c), SL(2,\C))$. 
We will keep the same notations for these functions and their lifts and/or pull-backs.

\bigskip

\begin{Thm}\label{thm:F}
With the above notations, consider the following local function 
\begin{eqnarray*}
F^c : X(\pi_1(\Sigma_\epsilon^c),SL(2,\C))  &\to & \R^{6g-6} = \R^{2|E|} \bigoplus_{v\in V} \R^{2d(v)-6} \\
\left[\varrho\right] &\mapsto &\bigg(\Big(\Re \tr_{\mu_e}(\varrho), \Im \tr_{\mu_e}(\varrho)\Big)_{e\in E}\ , \Big(\Im f^v_k(\varrho)\Big)_{v\in V , 1\leq k \leq 2d(v) -6}\bigg)
\end{eqnarray*}
Then in a neighborhood of $[\rho]$, which is a smooth point of $X(\pi_1(\Sigma_\epsilon^c),SL(2,\C))$,
the level sets of $F^c$ are local half-dimensional submanifolds. 
\end{Thm}

\begin{proof} The facts that the character variety $X(\pi_1(\Sigma_\epsilon^c),SL(2,\C))$ is smooth at $[\rho]$ and that its complex dimension is $6g-6$ are quite standard, see for instance \cite{Goldman1}.
To prove the theorem it is thus sufficient to show that the (complex) derivatives of the functions $tr_{\mu_e}$ and $f^v_k$ are $\C$-linearly independent on $T^*_{[\rho]} X(\pi_1(\Sigma_\epsilon^c),SL(2,\C))$; this will of course imply that the (real) derivatives of their imaginary and real parts are also independent, so that the map $F^c$ is locally an immersion.

As mentioned above, the family $(\mu_e)_{e\in E}$ splits $\Sigma_\epsilon^c$ into a family of punctured spheres $(S_v)_{v\in V}$, corresponding each to a vertex $v$ of $\Sigma^c$, and the induced representation on $S_v$ satisfies the hypotheses of Proposition \ref{propfunctions}. We will denote by $\mu^v_k$, $k=1\dots d(v)$, the simple closed curves corresponding to the boundary components of $S_v$; for each $k$ the image $i_v(\mu^v_k)$ is homotopic (up to orientation) to one curve of the family $(\mu_e)_{e\in E}$, but note that two boundary components of $S_v$ may correspond to the same curve in $\Sigma_\epsilon^c$. 

The next step is to rebuild $\Sigma_\epsilon^c$ by gluing together the punctured spheres $S_v$ obtained by cutting $\Sigma_\epsilon^c$.
Thus we construct a family $\Sigma_0,\Sigma_1,\ldots, \Sigma_{|E|} = \Sigma_\epsilon^c$, where $\Sigma_0$ is one of the punctured spheres $S_{v_0}$, and where $\Sigma_{l}$ is obtained from $\Sigma_{l-1}$ by performing one of the following operations:
\begin{enumerate}
\item gluing a punctured sphere $S_{v_l}$ to $\Sigma_{l-1}$ along one of its boundary components,
\item gluing $\Sigma_{l-1}$ along two of its boundary components.
\end{enumerate} 
The surface $\Sigma_l$ is therefore obtained as a gluing of a certain sub-family of the family of punctured spheres $(S_v)_{v\in V}$; let us denote by $V_l \subset V$ the corresponding subset of indices, i.e.~$\Sigma_l$ is obtained from the family $(S_v)_{v\in V_l}$. For each $l$, we have inclusion maps $i_l : \Sigma_l \hookrightarrow \Sigma_\epsilon^c$, $i_{v,l} : S_v \hookrightarrow \Sigma_l$ if $v \in V_l$, and $i_{l-1,l} : \Sigma_{l-1} \hookrightarrow \Sigma_{l}$ if  $l\geq 1$, satisfying the obvious compatibility relations.

For $v \in V_l$, the local functions $(f^v_k)_{1\leq k\leq 2d(v)-6}$ on $X(\pi_1(S_v),SL(2,\C))$ can be pulled back to functions on the character variety of $\pi_1(\Sigma_l)$ via $i_{v,l}^* : X(\pi_1(\Sigma_l),SL(2,\C)) \to X(\pi_1(S_v),SL(2,\C))$. On $\Sigma_l$ we also have a family $(\mu^l_k)_{1\leq k \leq I_l}$ of curves, corresponding to the boundary components of $\Sigma_l$ and to the curves along which the previous gluings have been done. Note that all the surfaces $\Sigma_l$, $0\leq l <|E|$, are compact and have a non empty boundary; their fundamental groups are therefore free groups and hence $R(\pi_1(\Sigma_l),SL(2,\C))\simeq SL(2,\C)^n$ is a smooth manifold. Since the representations $i_l^*\rho$ are irreducible, this means that the tangent spaces 
$T_{i_l^*\rho} X(\pi_1(\Sigma_l),SL(2,\C))$ are actually well-defined for all $l$ (the case $l=|E|$ has already been treated).

We will now prove by induction on $l$ that the family of functions $\big(\tr_{\mu^l_k}\big)_{1\leq k\leq I_l} \cup  \big(i_{v,l\,*}(f^v_k)\big)_{v\in V_l, 1\leq k \leq 2d(v)-6}$ has $\C$-linearly independent derivatives on $T^*_{[i_l^*\rho]} X(\pi_1(\Sigma_l),SL(2,\C))$. We already know that this is true for $l=0$ (Proposition \ref{propfunctions}). For the inductive step, there are two cases to consider, depending on the type of gluing.

\medskip

Case 1: $\Sigma_{l}$ is obtained as a gluing of $\Sigma_{l-1}$ and of a punctured sphere, denoted $S_{v_{l}}$, of the family $(S_v)_{v\in V}$; more precisely, a boundary component $\mu_1$ of $\Sigma_{l-1}$ is identified to a boundary component $\mu_2$ of $S_{v_l}$. The simple closed curve $\mu_1$ (resp. $\mu_2$) belongs to the family $(\mu^{l-1}_k)_{1\leq k \leq I_{l-1}}$ (resp. $(\mu^{v_l}_k)_{1\leq k\leq d(v_l)}$). 
We will denote by $\mu$ the resulting curve on $\Sigma_{l}$; it belongs to the family $(\mu^l_k)_{1\leq k \leq I_l}$.
We have the following commutative diagram, where all maps are inclusions:
$$\xymatrix @!0 @R=2em @C=4pc {
 & S_{v_l} \ar[rd]^{i_{v_l,l}} & \\
\mathbb{S}^1\ar[ru]^{j'} \ar[rd]_{j} & & \Sigma_l 
\\ & \Sigma_{l-1} \ar[ru]_{i_{l-1,l}} & & }$$

The fundamental group $\pi_1(\Sigma_{l})$ is obtained as an amalgamated product of $\pi_1(\Sigma_{l-1})$ and $\pi_1(S_{v_l})$, and $R(\pi_1(\Sigma_{l}), SL(2,\C))$ is equal to the fiber product of $R(\pi_1(\Sigma_{l-1}), SL(2,\C))$ and $R(\pi_1(S_{v_l}), SL(2,\C))$ over $R(\pi_1(\mathbb{S}^1), SL(2,\C))$:
$$R(\pi_1(\Sigma_{l}), SL(2,\C)) \simeq \left\{(\sigma,\tau) \in R(\pi_1(\Sigma_{l-1}), SL(2,\C))\times R(\pi_1(S_{v_l}), SL(2,\C)) \ ;\ j^*(\sigma) = j'^*(\tau)\right\}$$
where $j$ and $j'$ are the inclusion maps from the glued boundary component $\mu \simeq \mathbb{S}^1$ into $\Sigma_{l-1}$ and $S_{v_l}$ respectively. In particular, the restriction map 
\begin{eqnarray*}
X(\pi_1(\Sigma_{l}), SL(2,\C)) & \to & X(\pi_1(\Sigma_{l-1}), SL(2,\C))\times X(\pi_1(S_{v_l}), SL(2,\C))\\
\left[\varrho\right] & \mapsto & ([i_{l-1,l}^*\varrho],[i_{v_l,l}^*\varrho])
\end{eqnarray*}
is surjective onto the set $\{([\sigma],[\tau])\ ;\ \tr_{\mu_1}(\sigma)=\tr_{\mu_2}(\tau)\}$, at least locally away from $\tr_\mu = \pm2$ (recall that two elements of $SL(2,\C)$ whose traces are different from $\pm 2$ are conjugated if and only if their traces are equal). For the cotangent spaces, this yields an injective map
$$T^*_{[i_{l-1}^*\rho]}X(\pi_1(\Sigma_{l-1}), SL(2,\C))\oplus T^*_{[i_{v_l}^*\rho]}X(\pi_1(S_{v_l}), SL(2,\C)) /_{d\tr_{\mu_1}\sim d\tr_{\mu_2}} \to 
T^*_{[i_l^*\rho]}X(\pi_1(\Sigma_{l}), SL(2,\C)).$$
The fact that the functions $\big(\tr_{\mu^l_k}\big)_{1\leq k\leq I_l} \cup  \big(i_{v,l\,*}(f^v_k)\big)_{v\in V_l, 1\leq k \leq 2d(v)-6}$ have $\C$-linearly independent derivatives at $[i_l^*\rho]$ then follows easily from this, Proposition \ref{propfunctions} applied to $S_{v_l}$, and the independence of the derivatives of the similar functions on $T^*_{[i_{l-1}^*\rho]}X(\pi_1(\Sigma_{l-1}), SL(2,\C))$.

\medskip

Case 2: $\Sigma_l$ is obtained by gluing $\Sigma_{l-1}$ along two of its boundary components. Let us denote by $\mu_1=\mu^{l-1}_{k_l}$ and $\mu_2=\mu^{l-1}_{k'_l}$ the identified boundary components of $\Sigma_{l-1}$ (taken with matching orientations), and by $\mu=\mu^l_{k''_l}$ the resulting curve on $\Sigma_l$. The fundamental group of $\Sigma_{l}$ is obtained as an HNN-extension of $\pi_1(\Sigma_{l-1})$; more precisely, if $\<G|R\>$ is a presentation of $\pi_1(\Sigma_{l-1})$, and $\gamma_1, \gamma_2$ are two elements of $\pi_1(\Sigma_{l-1})$ corresponding to $\mu_1$ and $\mu_2$ respectively, then a presentation of $\pi_1(\Sigma_{l})$ is given by $\<G,t\ |\ R, t\gamma_1t^{-1} = \gamma_2\>$, and we have the following identification: 
\begin{align*}
R(\pi_1(\Sigma_{l}), SL(2,\C)) &\stackrel{\sim}{\to} \left\{(\sigma,B) \in R(\pi_1(\Sigma_{l-1}),SL(2,\C))\times SL(2,\C)\ ;\ B\sigma(\gamma_1)B^{-1} = \sigma(\gamma_2)\right\} \\
\varrho \quad &\mapsto \quad (i_{l-1,l}^*(\varrho), \varrho(t))
\end{align*}
In particular, the restriction map $i_{l-1,l}^* : X(\pi_1(\Sigma_{l}), SL(2,\C)) \to X(\pi_1(\Sigma_{l-1}),SL(2,\C))$ is surjective onto the set $\{[\sigma]\ ; \ \tr_{\mu_1}(\sigma) = \tr_{\mu_2}(\sigma)\},$ at least locally away from $\tr_\mu = \pm 2$. This induces on the cotangent spaces an injective map 
$$T^*_{i_{l-1}^*\rho}X(\pi_1(\Sigma_{l-1}), SL(2,\C)) /_{d\tr_{\mu_1}\sim d\tr_{\mu_2}} \to 
T^*_{i_l^*\rho}X(\pi_1(\Sigma_{l}), SL(2,\C)).$$
We can then deduce as before the independence of the derivatives of the functions $\big(\tr_{\mu^l_k}\big)_{1\leq k\leq I_l} \cup  \big(i_{v,l\,*}(f^v_k)\big)_{v\in V_l, 1\leq k \leq 2d(v)-6}$ at $[i_l^*\rho]$ from this injective map and the independence of the derivatives of the similar functions on $T^*_{[i_{l-1}^*\rho]}X(\pi_1(\Sigma_{l-1}), SL(2,\C))$.
\end{proof}

\bigskip

If the singular locus $\Sigma$ has a component that is a circle, then $\Sigma_\epsilon$ has a torus component. The corresponding case has been treated in \cite[Theorems 4.4 and 4.5]{HK} and \cite[Section 6.6.1]{Weiss2}; we will just quote the result. We keep the notations of Theorem \ref{thm:F}: $\rho$ is the holonomy representation of a closed, hyperbolic cone-$3$-manifold with cone angles smaller that $2\pi$, $\Sigma^c$ is a circle component of the singular locus, $\Sigma^c_\epsilon$ is a torus, boundary of a tubular neighborhood of $\Sigma^c$. The induced representation on $\pi_1(\Sigma_\epsilon^c)$ is also denoted by $\rho$, and $\mu$ is a simple closed curve on $\Sigma^c_\epsilon$ going around $\Sigma^c$, i.e.~it is a meridian curve of the torus.  

\medskip

\begin{Thm}[Torus case, \cite{HK,Weiss2}]\label{thm:Ftorus}\mbox{}\\ 
With the above notations, consider the function 
\begin{eqnarray*}
F^c : X(\pi_1(\Sigma_\epsilon^c),SL(2,\C))  &\to & \R^2\\
\left[\varrho\right] &\mapsto &\Big(\Re \tr_{\mu}(\varrho), \Im \tr_{\mu}(\varrho)\Big)
\end{eqnarray*}
Then in a neighborhood of $[\rho]$, which is a smooth point of $X(\pi_1(\Sigma_\epsilon^c),SL(2,\C))$,
the level sets of $F^c$ are local half-dimensional submanifolds. 
\end{Thm}

\section{Local deformations}\label{sec:locdef}

As usual, let $\bar M$ be a closed, orientable, hyperbolic cone-$3$-manifold with cone angles smaller than $2\pi$. Our goal in this section is to find local coordinates near $[\rho]$ on the character variety $X(\pi_1(M),SL(2,\C))$; this will be achieved by using the infinitesimal rigidity result of Section~\ref{subsec:infrig} to lift the functions $F^c$ defined above. Since these functions have natural geometric interpretations it will then easy to deduce the local parametrization of the space of cone-manifolds (Theorem \ref{thmStoker1}). The application to the Stoker problem is straightforward but necessitates some background, in particular the fact that the space of convex polyhedra with fixed combinatorial type is a manifold whose dimension is equal to the number of edges (Proposition \ref{prop:manpol}).

First we need to know that $X(\pi_1(M),SL(2,\C))$ is smooth near $[\rho]$. This is indeed the case according to the following theorem, which can be found in M. Kapovich's book \cite{Kap}; see also \cite{CullerShalen,ThurstonGeom}, and \cite[Section 6.7.1]{Weiss2}.

\medskip

\begin{Thm}[Smoothness of the Holonomy Representation, \cite{Kap}]\label{thm:Mirred}\mbox{}\\
Let $\rho$ be the holonomy representation of a closed, oriented, connected hyperbolic cone-$3$-manifold $\bar M$, with cone angles smaller than $2\pi$. Then $X(\pi_1(M),SL(2,\C))$ is smooth at $[\rho]$, and its real dimension is $2\tau - 3 \chi(\Sigma_\epsilon)$, where $\tau$ is the number of torus components in $\Sigma_\epsilon$.
\end{Thm}

\medskip

Let us denote by $(\Sigma^c)_{c\in \mathcal{C}}$ the components of $\Sigma$, the singular locus of $\bar M$, so $\Sigma = \coprod_{c \in \mathcal{C}} \Sigma^c$. The surface $\Sigma_\epsilon$, which is the boundary  of a tubular neighborhood of the singular locus, also decomposes as a union $\Sigma_\epsilon = \coprod_{c\in \mathcal{C}} \Sigma_\epsilon^c$. For each component of $\Sigma_\epsilon$, the inclusion map $i_c : \Sigma_\epsilon^c \hookrightarrow M$ induces a map $i_c^* : X(\pi_1(M), SL(2,\C)) \to X(\pi_1(\Sigma_\epsilon^c), SL(2,\C))$. We can apply Theorem \ref{thm:F} or Theorem \ref{thm:Ftorus} to each component of $\Sigma_\epsilon$: this gives a family of local functions $F^c : X(\pi_1(\Sigma_\epsilon^c), SL(2,\C)) \to \R^{n_c}$, where $n_c = 2$ if $\Sigma_\epsilon^c$ is a torus, or $n_c = 6g_c-6$ if $\Sigma_\epsilon^c$ has genus $g_c \geq 2$. 
These functions can be lifted to local functions $i_{c*} F_c$ on $X(\pi_1(M), SL(2,\C))$ in a neighborhood of $[\rho]$. Note that the real dimension of $X(\pi_1(M),SL(2,\C))$ is exactly equal to $\sum_{c\in \mathcal{C}} n_c$.

\medskip

\begin{Thm}
Let $\rho$ be the holonomy representation of a closed, oriented, connected hyperbolic cone-$3$-manifold $\bar M$, with singular locus $\Sigma = \coprod_{c\in \mathcal{C}} \Sigma^c$ and cone angles smaller than $2\pi$. Then the local function 
\begin{eqnarray*}
F = (i_{c*} F^c)_{c \in \mathcal{C}} : X(\pi_1(M), SL(2,\C)) & \to & \bigoplus_{c\in \mathcal{C}} \R^{n_c}\\
\ [\varrho] & \mapsto & \big(F^c(i_c^* [\rho])\big)_{c \in \mathcal{C}}
\end{eqnarray*}
is a coordinate chart of $X(\pi_1(M),SL(2,\C))$ in a neighborhood of $[\rho]$.
\end{Thm}

\begin{proof}
Let us consider the restriction map:
\begin{eqnarray*}
r : X(\pi_1(M), SL(2,\C)) & \to & \bigtimes_{c\in \mathcal{C}} X(\pi_1(\Sigma_\epsilon^c), SL(2,\C)) \\
\ [\varrho] & \mapsto & (i_c^* [\varrho])_{c \in \mathcal{C}} 
\end{eqnarray*}   
We have seen in Section \ref{sec:defotheory}  that the Zariski tangent space at $[\rho]$ (resp. $i_c^* [\rho]$) of $X(\pi_1(M),SL(2,\C))$ (resp. $X(\pi_1(\Sigma_\epsilon^c), SL(2,\C))$) is identified with the first cohomology group $H^1(M;\E)$ (resp. $H^1(\Sigma_\epsilon^c;\E)$); and since $[\rho]$ (resp. $i_c^* [\rho]$) is a smooth point, this is also the usual tangent space. The tangent map $Dr$ at $[\rho]$ is thus identified with the natural map 
$$H^1(M;\E) \to \bigoplus_{c\in \mathcal{C}} H^1(\Sigma^c_\epsilon;\E) = H^1(\Sigma_\epsilon;\E) = H^1(U_\epsilon;\E).$$
According to Proposition \ref{prop:H1}, this is an injective map, with half-dimensional image. Hence $r$ is an immersion near $[\rho]$, and we just have to prove that its image is locally transverse to the level sets of the function $F$. Let us denote by $L$ the level set of the function $F$ passing through $[\rho]$; its tangent space $T_{[\rho]}L$ is identified to a subspace of $H^1(U_\epsilon;\E)$. It is sufficient to show that the image of $Dr$ is in direct sum with $T_{[\rho]}L$; because of dimensions, it is actually enough to prove that $T_{[\rho]} L$ and the image of $Dr$, i.e.~the image of $H^1(M;\E)$ in $H^1(U_\epsilon;\E)$, have trivial intersection.

By the definition of the function $F$, the elements of $T_{[\rho]}L$ are in one-to-one correspondence with the infinitesimal deformations that preserve the holonomy classes of all the meridian curves $\mu_e$, $e\in E$, and preserve the fact that the induced holonomy on the vertices' links is spherical, i.e.~has values in a maximal compact subgroup. We have introduced in Section \ref{subsec:infrig} ``standard'' closed $\E$-valued forms on $U_\epsilon$; more precisely, for each singular edge $e$ we have defined two $1$-forms $\omega^e_\lambda = d\phi \wedge \sigma_{\d / \d z}$ and $\omega^e_\tau = d\phi \wedge \sigma_{\d / \d\theta}$. It is clear that their cohomology classes belong to $T_{[\rho]}L$ since they preserve the induced holonomy on each links and on each edge's meridian. Similarly, we have seen that any angle-preserving infinitesimal deformation of the spherical structure on the link $L_v$ of a singular vertex $v$ (which is homeomorphic to the punctured sphere $S_v$) corresponds to a closed $\E$-valued $1$-form on $U_\epsilon$, vanishing near the singular edges. The class of such a form also clearly belongs to $T_{[\rho]}L$. Using Proposition \ref{propfunctions}, we see that the space of angle-preserving infinitesimal deformation of $L_v$ has (real) dimension $2d(v)-6$. So we can find $2d(v)-6$ such forms $\omega^v_1,\ldots, \omega^v_{d(v)-6}$ whose classes are linearly independent in $H^1(U_v;\E)$, where $U_v \subset U_\epsilon$ is the regular part of a cone neighborhood of $v$. 

Because of the dimension of $L$, if the family $(\omega^v_k)_{v\in V, 1\leq k \leq 2d(v)-6} \cup (\omega^e_\lambda, \omega^e_\tau)_{e \in E}$ is linearly independent in $H^1(U_\epsilon;\E)$ then it forms a basis of $T_{[\rho]}L$. So, suppose there exists a linear combination of these forms which is a coboundary:
$$\sum_{v,k} \lambda^v_k \omega^v_k + \sum_e (\lambda^e_\lambda \omega^e_\lambda +\lambda^e_\tau \omega^e_\tau) = ds$$ for some section $s$ of $\E$. Over $U_v$, we get $\sum_k \lambda^v_k \omega^v_k = ds$ (the other forms being supported away from $U_v$), which implies $\lambda^v_k = 0$ for $1\leq 2d(v)-6$ because the forms $\omega^v_k$ are independent in $H^1(U_v;\E) \simeq T_{[i_{v}^*\rho]} X(\pi_1(S_v),SL(2,\C))$. So $ds=0$ over $U_v$, but since $i_v^* \rho$ is irreducible there exists no non-zero constant section of $\E$ over $U_v$. Therefore $s_{|U_v}=0$, and this being true for any $v$, finally $s=0$ away from the support of the family $(\omega^e_\lambda, \omega^e_\tau)_{e \in E}$.
Now on a neighborhood of a singular edge $e$, we have $ds =  \lambda^e_\lambda \omega^e_\lambda +\lambda^e_\tau \omega^e_\tau = \lambda^e_\lambda d\phi \wedge \sigma_{\d / \d z} + \lambda^e_\tau  d\phi \wedge \sigma_{\d / \d\theta}$. Integrating $ds$ along a path parallel to $e$, we obtain $0 =  \lambda^e_\lambda \sigma_{\d / \d z} + \lambda^e_\tau \sigma_{\d / \d\theta}$, thus $\lambda^e_\lambda = \lambda^e_\tau =0$. So the family  $(\omega^v_k)_{v\in V, 1\leq k \leq 2d(v)-6} \cup (\omega^e_\lambda, \omega^e_\tau)_{e \in E}$ is indeed linearly independent in $H^1(U_\epsilon;\E)$, hence a basis of $T_{[\rho]} L$.

We can now finish the proof. Let $[\omega] \in H^1(M;\E)$ be such that $Dr([\omega])$ belongs to $T_{[\rho]}L$. So $Dr([\omega])$ has a representative $\omega'$ which is a linear combination of the above forms. But since $Dr$ is just the restriction to $U_\epsilon$, this means that $[\omega]$ also has a representative $\omega$ which is over $U_\epsilon$ a linear combination of the above forms. The corresponding infinitesimal deformation $h$ of the metric tensor then satisfies the hypotheses of Theorem \ref{thm:rigL2}, so that $h$ is trivial; thus $\omega$ is a coboundary and $[\omega] = 0$. Consequently $T_{[\rho]}L$ and the image of $Dr$ are in direct sum, which implies that the image of $r$ and the level sets of the function $F$ are locally transverse.   
\end{proof}

\bigskip
Now that we have established the existence of this local coordinate system, we are in a position to prove the two main results of this article:

\begin{Thm} \label{thmStoker1}
Let $\bar M$ be a closed, orientable, hyperbolic cone-$3$-manifold with singular locus $\Sigma$ and cone angles smaller than $2\pi$. Then the tuple $(\alpha_e)_{e\in E}$ of cone angles gives a local parametrization of the space of closed hyperbolic cone-$3$ manifolds with singular locus $\Sigma$ in a neighborhood of $\bar M$.
\end{Thm}

\begin{proof} As usual, let $\rho$ be the holonomy representation of $M$. We know that $\rho$ is irreducible (Corollary \ref{cor:rhoirred}), so according to Theorem \ref{thm:Goldman} (from \cite{Goldman4}), the space of (equivalence classes of) hyperbolic structures on the regular part $M$ of $\bar M$ in a neighborhood of the given cone-manifold structure is homeomorphic to a neighborhood of $[\rho]$ in $X(\pi_1(M),SL(2,\C))$. By the preceding corollary, it admits as local coordinates the functions $(\Re \tr_{\mu_e}, \Im \tr_{\mu_e})_{e\in E}$ and $(\Im f^v_k)_{v\in V, 1\leq k \leq 2d(v)-6}$. But most of these neighboring hyperbolic structures are not cone-manifolds ones. 
Locally, the set of cone-manifold structures with singular locus $\Sigma$ on $M$ corresponds to the set of representations such that the holonomy of any vertex's link is contained in a maximal compact subgroup of $SL(2,\C)$, i.e.~is contained in $SU(2)$ up to conjugacy, and such that the holonomies of the edges are elliptic isometries.
As we have seen in Proposition \ref{propfunctions}, this is equivalent to the conditions $\Im \tr_{\mu_e} = 0$, $e \in E$ and $f^v_k = 0$, $v \in V, 1\leq k \leq 2d(v)-6$. Therefore the space of cone-manifold structures with singular locus $\Sigma$ is locally parametrized by the tuple of the (real) traces of the singular edges' holonomy. Now if $\alpha$ is the cone angle of an edge, then the corresponding trace is equal to $\pm 2 \cos (\alpha/2)$;  thus the cone angle tuple yields another valid parametrization. 
\end{proof}

Remark that if $\Sigma$ has some vertices of valency greater than $3$, then $M$ most certainly admits other nearby cone-manifold structures than those mentioned above, but their singular loci will be different from $\Sigma$. As discussed in the introduction, this is because some vertices may ``split'' into several lower valency vertices, see Fig.~\ref{fig1}.

\bigskip

We need some more definitions before turning our attention to the proof of the Stoker problem. We will say that two convex polyhedra $\mathcal{P}_1$ and $\mathcal{P}_2$ have the same combinatorial type if there exists an oriented homeomorphism $\mathcal{P}_1 \to \mathcal{P}_2$ which sends faces to faces, edges to edges and vertices to vertices. Equivalently, $\mathcal{P}_1$ and $\mathcal{P}_2$ have the same combinatorial type if there exists a bijection $f$ between the sets of vertices of $\mathcal{P}_1$ and $\mathcal{P}_2$, such that two vertices bound an edge in $\mathcal{P}_1$ if and only if their images bound an edge in $\mathcal{P}_2$, and a family of vertices of $\mathcal{P}_1$ is the set of vertices of a face if and only if its image is the set of vertices of a face of $\mathcal{P}_2$. Such a map will be called a marking of $\mathcal{P}_2$ by $\mathcal{P}_1$. 

Now let us fix a convex polyhedron $\mathcal P$. A marked polyhedron having the combinatorial type of $\mathcal P$ is a couple $(\mathcal{Q},f)$ where $f$ is a marking of $\mathcal{Q}$ by $\mathcal{P}$. We define $\overline{Pol}(\mathcal P)$ as the set of convex marked polyhedra having the combinatorial type of $\mathcal P$. The direct isometry group of the ambient space acts freely on $\overline{Pol}(\mathcal P)$; the quotient is denoted by $Pol(\mathcal P)$. The reason for introducing marked polyhedra is that the space of congruence classes of convex polyhedra with the combinatorial type of $\mathcal P$ is generally an orbifold, whereas its ramified cover $Pol(\mathcal P)$ is a smooth manifold.

\begin{Prop}\label{prop:manpol}
Let $\mathcal P$ be a closed (strictly) convex polyhedron in $\H^3$, having $|E|$ edges. Then $Pol(\mathcal P)$ is a manifold of dimension $|E|$.
\end{Prop}

This fact seems to be well-known but oddly enough we could not find a correct reference for it, even in the case of Euclidean polyhedra; most textbooks restrict to the easy cases where all faces are triangle or all vertices are trivalent. Since this is not the main goal of this article, the following proof can be skipped by the reader. Note however that a simple calculation yields $|E|$ as the correct dimension. Let us denote by $F$ (resp. $E$ and $V$) the set of faces (resp. edges and vertices) of $\mathcal P$. Each face has three degrees of freedom, and each vertex $v$ gives $d(v)-3$ conditions because its $d(v)$ incident faces must meet at $v$. So heuristically, the dimension of $\overline{Pol}(\mathcal P)$ is $3|F| - \sum_{v\in V} (d(v)-3) = 3|F| -  \sum_{v\in V} d(v) + 3|V| $. But $d(v)$ is also the number of edges incident at $v$, and since each edge is incident to exactly two vertices, $\sum_{v \in V} d(v) = 2|E|$. So we obtain that the dimension of $\overline{Pol}(\mathcal P)$ is $3|F|+3|V|-2|E|$, which is equal to $|E|+6$ since the Euler characteristic of a convex polyhedron is $2$ (i.e.~$|F|+|V|-|E| = 2$). Finally, since the action of $\Isom^+(\H^3)$ on $\overline{Pol}(\mathcal P)$ is free and proper, the quotient $Pol(\mathcal P)$ has dimension $\dim \overline{Pol}(\mathcal P) - \dim \Isom^+(\H^3) = |E|$. The goal of the proof below is to show that this computation is indeed correct. 

\begin{proof}
We will begin by proving this result for convex polyhedra in Euclidean space. We recall that a convex polyhedron can be defined as a (bounded) finite intersection of half-spaces. In particular, a face can be identified with its supporting half-space. The set of all half-spaces is a three-dimensional manifold,  which we will realize as the $3$-cylinder $C^3 = \mathbb S^2 \times \R \subset \R^4$ (a half-space of equation $ax+by+cz+d \geq 0$ corresponds to the point $(a/n,b/n,c/n,d/n)$ where $n=\sqrt{a^2+b^2+c^2}$). With this embedding, it is clear that the intersection of the boundary planes of four half-spaces $p_1,\dots,p_4$ is non-empty if and only if $\det (p_1,p_2,p_3,p_4) = 0$. It is also clear that a point $x\in \R^3$ lies on the boundary plane of a half-space $p$ if and only if $p\cdot (x,1) = 0$ (for the usual dot product on $\R^4$). Let us define on $\R^4$ the cross-product of $3$ vectors by asking that $u\wedge v \wedge w$ is the unique vector such that $x \cdot (u\wedge v \wedge w) = \det (x,u,v,w)$ for all $x \in \R^4$. The previous remarks imply that if the boundary planes of three half-spaces $p_1, p_2, p_3$ meet in a single point $v \in \R^3$, then $p_1 \wedge p_2 \wedge p_3 = (c\,v, c)$ for some $c \in \R^*$.

For each vertex $v$ of the polyhedron $\mathcal P$, let $\mathrm{f}^v_1,\mathrm{f}^v_2,\dots,\mathrm{f}^v_{d(v)}$ be the faces incident to $v$; the order in which we consider them will be discussed later. An element of $\overline{Pol}(\mathcal P)$, that is, a convex Euclidean marked polyhedron having the combinatorial type of $\mathcal P$, is then a collection $\left(p(\mathrm f)\right)_{\mathrm f \in F} \in (C^3)^{|F|}$ of distinct half-spaces satisfying the following conditions: 
\begin{enumerate}
\item (incidence) for each vertex $v\in V$, the boundary planes of $p(\mathrm{f}^v_1), \dots, p(\mathrm{f}^v_{d(v)})$ intersect in a single point;
\item (minimality) the intersection is non-redundant.
\end{enumerate}
Let $\mathcal P_0 = (p(\mathrm f))_{\mathrm f \in F}$ be an element of $\overline{Pol}(\mathcal P)$. The properties that an intersection of half-spaces is non-redundant and that three planes meet at a single point are clearly open, so that in a neighborhood of $\mathcal P_0$, the two conditions can be reformulated as: for each vertex $v\in V$ such that $d(v) \geq 4$, for each $i \in \{1, \dots, d(v)-3\}$, the boundary planes of $p(\mathrm{f}^v_i), p(\mathrm{f}^v_{i+1}), p(\mathrm{f}^v_{i+2})$ and $p(\mathrm{f}^v_{i+3})$ intersect in a single point, i.e.~$\det(p(\mathrm{f}^v_i), p(\mathrm{f}^v_{i+1}), p(\mathrm{f}^v_{i+2}), p(\mathrm{f}^v_{i+3}))=0$. It remains now to check that the functions $\psi_{v,i} = \det(p(\mathrm{f}^v_i), p(\mathrm{f}^v_{i+1)}, p(\mathrm{f}^v_{i+2}), p(\mathrm{f}^v_{i+3}))$ defined on $(C^3)^{|F|}$ have linearly independent derivatives at $\mathcal P_0$. Obviously, this is true if there is no vertex of valency greater than $3$; for the general case we will need the following combinatorial result.

\begin{Lem}\label{lem:comb}
Let $G$ be a connected planar graph with no mono- or bivalent vertex. Let $V'$ be a subset of vertices such that the valency of every element of $V'$ is greater or equal to $4$. If $V'$ is not empty, then there exists a face having at least $1$ and no more than $3$ vertices belonging to $V'$.
\end{Lem}

\begin{proof}
As before, we will denote by $V$, $E$ and $F$ the sets of vertices, edges and faces of $G$; it satisfies the Euler formula $|V|-|E|+|F|=2$. Let us first assume that $V''=V\setminus V'$ is exactly the set of trivalent vertices, and that every face has at least one vertex in $V'$. The valency of a vertex $v$ is still denoted by $d(v)$, and the number of vertices (or equivalently of edges) of a face $\mathrm f$ is denoted by $d(\mathrm f)$. Then $2|E| = \sum_{v\in V} d(v) \geq 4 |V'| + 3 |V''| $. By contradiction, if all faces have at least $4$ vertices in $V'$, then we also have $2|E| = \sum_{\mathrm f\in F}  d(\mathrm f) \geq  4 |F| + 3 |V''|$, where the second term is the contribution of the trivalent vertices. So $4|E| \geq 4 |V'| + 6 |V''| + 4 |F| \geq 4 |V| + 4|F|$, which contradicts Euler formula. Now, if $V'$ is not empty, we can always reduce the general case to the previous one using simple graph moves, see Fig.~\ref{fig:moves}: firstly, every face whose vertices all belong to $V''$ is collapsed to a single vertex; secondly, every vertex $v$ in $V''=V\setminus V'$ of valency $d(v)>3$ is split into $d(v)-2$ trivalent vertices.

\begin{figure}[h]
\centering
\includegraphics{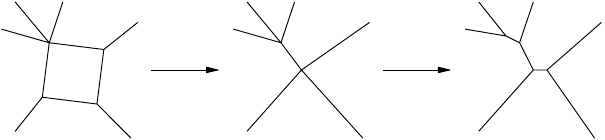}
\caption{Shrinking of a face and splitting of vertices}\label{fig:moves}
\end{figure}
\end{proof}

Let $V'$ be the set of vertices of $\mathcal P$ of valency greater than $3$, and suppose that $V'$ is not empty (otherwise there is nothing to prove). Suppose now that there exists a linear dependency $\sum_{v\in V'} \sum_{i=1}^{d(v)-3} \lambda_{v,i}\,d\psi_{v,i} = 0$ on the cotangent space of $(C^3)^{|F|}$ at $\mathcal P_0$. According to the previous lemma applied to the graph of $\mathcal P$, there exists a face $\mathrm f$ with at least $1$ and at most $3$ vertices in $V'$. Without loss of generality we can assume that $\mathrm f = \mathrm f^v_1$ for each of its vertices $v$, since the order of the faces at each vertex is arbitrary. Let $v_1$, $v_2$ and $v_3$ be its vertices belonging to $V'$ (the proof is similar if $\mathrm f$ has only one or two vertices in $V'$). If we evaluate the above dependency relation on a first-order deformation $(0,\dots,0,\dot p_\mathrm f,0,\dots,0)$ that only affects $p(\mathrm f)$, we obtain that 
\begin{multline*}
0  = \sum_{i=1}^3 \lambda_{v_i,1} d\psi_{v_i,1}(0,\dots,0,\dot p_\mathrm f,0,\dots,0) =  \sum_{i=1}^3 \lambda_{v_i,1} \det(\dot p_\mathrm f,p(\mathrm f^{v_i}_2),p(\mathrm f^{v_i}_3),p(\mathrm f^{v_i}_4)) 
\\ =  \left( \sum_{i=1}^3 \lambda_{v_i,1}\,p(\mathrm f^{v_i}_2) \wedge p(\mathrm f^{v_i}_3) \wedge p(\mathrm f^{v_i}_4) \right) \cdot \dot p_\mathrm f.
\end{multline*}
This equality holds for all $\dot p_\mathrm f \in T_{p(\mathrm f)} C^3$, but it is also satisfied by $p(\mathrm f)$, so that actually  $\sum_{i=1}^3 \lambda_{v_i,1}\,p(\mathrm f^{v_i}_2) \wedge p(\mathrm f^{v_i}_3) \wedge p(\mathrm f^{v_i}_4) =0$. But we have seen that $p(\mathrm f^{v_i}_2) \wedge p(\mathrm f^{v_i}_3) \wedge p(\mathrm f^{v_i}_4)$ is a multiple of $(x_i,1)$, where $x_i\in \R^3$ is the position of the vertex $v_i$ in $\mathcal P_0$. Since the three vertices are not aligned in $\mathcal P_0$, the family $\{(x_1,1),(x_2,1),(x_3,1)\}$ is linearly independent, and thus $\lambda_{v_1,1}=\lambda_{v_2,1}=\lambda_{v_3,1}=0$. 

For the next step, we apply Lemma \ref{lem:comb} to the graph obtained by shrinking the face $\mathrm f$ (as in Fig.~\ref{fig:moves}), and to the set of vertices that still have valency greater than $4$, except possibly the new vertex corresponding to $\mathrm f$. This gives us a second face $\mathrm f' \in F$, and we can assume that it is equal to $\mathrm f^v_1$ or $\mathrm f^v_2$ for each of its vertices. Applying a first-order deformation that only affects $p(\mathrm f')$ to the linear dependency will yield as above the vanishing of (up to three) other coefficients, and we can continue this process to show that the family $(d\psi_{v,i})$ is indeed linearly independent. This ends the proof that in the Euclidean case, $\overline{Pol}(\mathcal P)$ is a manifold of dimension $|E|+6$.

Now let us go back to $\H^3$. We will use the projective (or Klein) model to embed $\H^3$ as the open unit ball of $\R^3$; this embedding is not conformal, but maps geodesic lines and planes of $\H^3$ to (portions of) straight lines and planes. In particular, it identifies convex hyperbolic polyhedra with convex Euclidean polyhedra contained in the open unit ball, and of course it preserves the combinatorial type. This shows that the above conclusion is also valid for convex hyperbolic polyhedra. 
\end{proof}

\bigskip

\begin{Thm} \label{thmStoker2}
Let $\mathcal P$ be a closed (strictly) convex polyhedron in $\H^3$, having $N$ edges. Then the tuple $(\alpha_1,\ldots,\alpha_N)$ of dihedral angles gives a local parametrization of $Pol(\mathcal P)$.
\end{Thm}

\begin{proof}
Let us denote by $\bar M = D(\mathcal P)$ the double of $\mathcal P$; this is a closed hyperbolic cone-$3$-manifold, whose cone angles are smaller than $2\pi$ because $\mathcal P$ is convex, and its singular locus $\Sigma$ corresponds to the graph formed by the edges and the vertices of $\mathcal P$. Let $C(M,\Sigma)$ be the space of cone-manifold structures with singular locus $\Sigma$ on the regular part $M$ of $D(\mathcal P)$. The double construction yields an injective map 
\begin{equation} Pol(\mathcal P) \to C(M,\Sigma) \label{eqdouble} \end{equation}
Now according to Theorem \ref{thmStoker1} the space $C(M,\Sigma)$ is near $D(\mathcal P)$ a manifold whose dimension is equal to the number of its singular edges. Hence \eqref{eqdouble} is near $\mathcal P$ an injective map between two manifolds of the same dimension, and thus a local diffeomorphism. The local parametrization of $C(M,\Sigma)$ by the cone angles therefore yields a parametrization of $Pol(\mathcal P)$ by the dihedral angles.
\end{proof}

\bigskip
\noindent {\bf Acknowledgments.}
The author would like to thank Rafe Mazzeo, Steve Kerckhoff and Olivier Guichard for their help and advices during the elaboration of this article.

\end{document}